\documentclass{emsprocart}
\usepackage{tikz}
\usepackage{amssymb,relsize}
\usepackage{makeidx}

\usepackage{amscd,mathrsfs,eucal}

\contact[koen.thas@gmail.com]{Koen Thas, Department of Mathematics, Ghent University, Krijgslaan 281 | S25, Ghent, Belgium}

\newcommand{\R}{\mathbb{R}}

\newcommand{\id}{\mathbf{1}}

\newcommand{\fp}{\mathfrak{p}}
\newcommand{\fm}{\mathfrak{m}}
\newcommand{\Div}{\mathrm{Div}}
\newcommand{\divf}{\mathrm{div}}
\def\bp{{\texttt{Blpr}}}
\def\bpspaces{{\texttt{LocBlprSp}}}

\newcommand{\Sch}{\texttt{Sch}}
\newcommand{\BSch}{\Sch_{\Fun}}

\def\toptimes{\times^\top}
\def\top{{\textup{top}}}

\DeclareMathOperator*{\Rprod}{\mathpalette\doubleprod\relax}

\newcommand{\Z}{\mathbb{Z}}

\newcommand{\F}{\mathbb{F}}
\newcommand{\Fun}{\mathbb{F}_1}

\newcommand{\cO}{\ensuremath{\mathscr{O}}}

\newcommand{\mD}{\ensuremath{\mathcal{D}}}

\newcommand{\mC}{\ensuremath{\mathcal{C}}}

\newcommand{\mO}{\ensuremath{\mathscr{O}}}

\newcommand{\mP}{\ensuremath{\mathcal{P}}}

\def\doubleprod#1#2{\ooalign{$#1\prod$\cr$#1\coprod$\cr}}
\newcommand{\norm}[1]{\left| #1 \right|}

\newtheorem{theorem}{Theorem}[section]

\newtheorem{lemma}[theorem]{Lemma}
\newtheorem{proposition}[theorem]{Proposition}

\theoremstyle{definition}

\newtheorem{remark}[theorem]{Remark}

\newcommand{\wis}[1]{{\text{\em \usefont{OT1}{cmtt}{m}{n} #1}}}
\def\Spec{\wis{Spec}}

\makeindex

\title[A taste of Weil theory]{A taste of Weil theory in characteristic one}

\author[Koen Thas]{Koen Thas}

\begin{document}
\setcounter{page}{379}

\begin{abstract}
In this very short and sketchy chapter, we draw some pictures on the arithmetic theory of $\Fun$. 
\end{abstract}



\maketitle
\tableofcontents

\newpage
\section{Introduction}

In \cite[p. 259]{Haran}, Shai Haran writes the following ...

\begin{quote}
``It will not be an exaggeration to say that the greatest mystery of arithmetic is the simple fact that 
\begin{equation}
\Z \otimes \Z \cong \Z,
\end{equation}
or, equivalently, that from the point of view of Algebraic Geometry,
\begin{equation}
\Spec(\Z) \times \Spec(\Z) \cong \Spec(\Z),
\end{equation}
i.e., the surface reduces to the diagonal!''
\end{quote}

Nevertheless, as Haran states, for functions $f, g: \R^+ \mapsto \R$ which are smooth and compactly supported, to be thought of as representing ``Frobenius divisors'' on this nonexisting surface, one could define their intersection number as 
\begin{equation}
\Bigl\langle f, g \Bigr\rangle := W(f * g^{*}),
\end{equation}
where $W(\cdot)$ will be defined further on in this chapter, and associating to such a function $f$ a real number $h^0(f)$ (as below, in \S\S \ref{secHar}) satisfying the three properties stated in that same subsection, one will find the solution of the classical Riemann Hypothesis through a characteristic $0$ version of Weil's Fundamental Inequality. Quoting  Haran again:
\begin{quote}
``Ergo our main point is: a two dimensional Riemann-Roch for $\Spec(\Z)$ may very well exist!''
\end{quote}

\noindent
This very discussion could be taken as the ``definition'' of ``Absolute Arithmetic,'' which is the subject of this final and very short chapter.

\medskip
\subsection{Some questions}

\bigskip
Instead of considering the arithmetic zeta function of $\Spec(\Z)$, being
\begin{equation}
\zeta_{\Spec(\Z)}(s) = \prod_{p\ \mbox{prime}}\zeta_p(s) = \prod_{p\ \mbox{prime}}\frac{1}{1 - p^{-s}},
\end{equation}
we look for a ``compactified version''  $\zeta_{\overline{\Spec(\Z)}}(s)$, 
\begin{eqnarray}
\zeta_{\overline{\Spec(\Z)}}(s) &= \prod_{\{p\ \mbox{prime}\}\cup\{\infty\}}\zeta_p(s) &=  \underbrace{\pi^{-s/2}\Gamma(\frac{s}{2})} \cdot  \Bigl(\zeta_{\Spec(\Z)}(s)\Bigr).\nonumber \\
& &\zeta_{\infty}(s)\mbox{--factor}\nonumber \\
\end{eqnarray}

\bigskip
\textsc{Question}.\quad {\em Can one define a projective ``curve'' $\mC := \overline{\Spec(\Z)}$ over $\Fun$ whose zeta function $\zeta_{\mC}(s)$
is the complete Riemann zeta function}
\begin{equation}
\zeta_{\mathbb{Q}}(s) = \pi^{-s/2}\Gamma(\frac{s}{2})\zeta(s)?\footnote{We omit the factor $\frac{1}{\sqrt{2}}$ for now.}
\end{equation}

There are two essential problems to solve (at first sight):

\subsubsection{Cat}
Find the right setting in which we can see $\overline{\Spec(\Z)}$ as a projective curve | in one way or another.

\subsubsection{Zeta}
Obtain the desired zeta function identity for this curve.\\

Of course, finding a deeper base $\Upsilon$ over which $\overline{\Spec(\Z)}$ defines an object which, with an adapted zeta function over $\Upsilon$,
agrees with {\bf Zeta},  would be a good start as well. (Probably anything is.) 

As we have seen, over the base $\Upsilon$, we want to be able to define the surface $\Spec(\Z) \times \Spec(\Z)$. We introduce it as the ``next'' problem to solve.

\subsubsection{Prod}
Find the right setting in which we can see $\Spec(\Z) \times_{\Upsilon} \Spec(\Z)$ as a surface | in one way or another (and hopefully in agreement with {\bf Cat}).\\

In any case, the expression
\begin{equation}
\Spec(\Z) \times_{\Upsilon} \Spec(\Z) \not\cong \Spec(\Z)
\end{equation}
might be a good thing to start with.

\medskip
\subsection{Deninger's formula | an answer to {\bf Cat}}

We reprise the discussion concerning Deninger's formula (cf. the second chapter of the author).\\

Recall from the author's second chapter in this book that Deninger (in a series of works \cite{Deninger1991,Deninger1992,Deninger1994}) gave a description of conditions on a conjectural category of motives that would admit a translation of Weil's proof of the Riemann Hypothesis for function fields of projective curves over finite fields $\F_q$ to the hypothetical curve $\overline{\Spec(\Z)}$. In particular, he showed that the following formula would hold:

\begin{equation}
\label{Denform}
\zeta_{\overline{\Spec(\Z)}}(s) = 2^{-1/2}\pi^{-s/2}\Gamma(\frac s2)\zeta(s)  
		=  \mathlarger{\frac{\Rprod_\rho\frac{s - \rho}{2\pi}}{\frac{s}{2\pi}\frac{s - 1}{2\pi}}} \overset{?}{=} \nonumber \\
\end{equation}
	\begin{eqnarray} 	
	 \frac{\mbox{\textsc{Det}}\Bigl(\frac 1{2\pi}(s\cdot\id - \Theta)\Bigl| H^1(\overline{\Spec(\Z)},*_{\mathrm{abs}})\Bigr.\Bigr)}{\mbox{\textsc{Det}}\Bigl(\frac 1{2\pi}(s\cdot\id -\Theta)\Bigl| H^0(\overline{\Spec(\Z)},*_{\mathrm{abs}})\Bigr.\Bigr)\mbox{\textsc{Det}}\Bigl(\frac 1{2\pi}(s\cdot\id - \Theta)\Bigl| H^2(\overline{\Spec(\Z)},*_{\mathrm{abs}})\Bigr.\Bigr)}, 
	\end{eqnarray}
where: 
\begin{itemize}
\item
$\Rprod$ is the infinite {\em regularized product}; 
\item
similarly
$\mbox{\textsc{Det}}$ denotes the {\em regularized determinant} | a determinant-like function of operators on infinite dimensional vector spaces; 
\item
$\Theta$ is an ``absolute'' Frobenius endomorphism;
\item
the $H^i(\overline{\Spec(\Z)},*_{\mathrm{abs}})$ are certain cohomology groups, and
\item
the $\rho$s run through the set of critical zeros of the classical Riemann zeta function. 
\end{itemize}
In the formula displayed above, $\Spec(\Z)$ is compactified to $\overline{\Spec(\Z)}$ in order to see it as a projective curve.
All the details can be found in the aforementioned chapter of the author.

The right-hand side of (\ref{Denform}) has the form of a weight decomposition in which the denominator has factors corresponding to  zeta functions 
of a point and an affine line over some base $\Upsilon$, and the numerator is the zeta-factor which distinguishes the ``curve'' from a projective line.

After work of  Kurokawa \cite{Kurokawa1992} and Manin \cite{Manin}, the concensus was born that the factors, in obvious notation, were to be seen as:
\begin{itemize}
 \item
 $h^0$: ``the absolute point,'' $\Spec(\Fun)$; 
 \item
 $h^1$: the numerator, and 
 \item 
 $h^2$: ``the absolute Lefschetz motive,'' that is, the affine line over $\Fun$,
 \end{itemize}
  with zeta functions
	\begin{equation}
	\label{eqzeta}
		\zeta_{h^w}(s) \ = \ \mbox{\textsc{Det}}\Bigl(\frac 1{2\pi}(s\cdot\id-\Theta)\Bigl| H^w(\overline{\Spec(\Z)},*_{\mathrm{abs}})\Bigr.\Bigr) 
	\end{equation}
for $w=0,1,2$.

\bigskip
\subsection*{About this chapter}

Not much is known on any of the questions we mentioned so far (beyond what was already mentioned earlier in the present book). In this chapter, which the reader perhaps wants to see as an appendix to the ``real'' body of this monograph, rather than an independent chapter (due to the simple fact that this part of the story at present still needs to mature), I will try to make some comments on {\bf Cat}, and especially {\bf Zeta} and {\bf Prod}. 

It  (= the chapter) consists of three parts: first of all, I want to make some notes on Weil's proof of the Riemann Hypothesis for function fields of projective curves over finite fields, expanding what was already mentioned in quotes by Shai Haran, to make the goal of the game more accessible. Secondly, I want to mention recent results of Connes and Consani \cite{ConCon} and Kurokawa and Ochiai \cite{KuroOchi} which deal with the counting function of the hypothetical curve $\overline{\Spec(\Z)}$, in relation to {\bf Zeta}. 
Finally, I want to summarize some views on the looks of $\overline{\Spec(\Z)}$ (over $\Fun$). In that part, I will recall some relevant pictures which can be found in more detail elsewhere in this book, and I will also (wildly) speculate on this subject.

\medskip
\subsection*{Acknowledgment}

I want to thank
Christopher Deninger and Nobushige Kuro\-ka\-wa for several highly helpful communications on the subject of this chapter.

\medskip
\section{Notes on the Weil conjectures in dimension $1$}

Let $\F_q$ be a finite field and $\overline{\F_q}$ an algebraic closure of $\F_q$.
Assume that $X$ is a projective scheme over $\F_q$ such that $X \times_{\Spec(\F_q)}\Spec(\overline{\F_q})$ is irreducible and nonsingular.
In \cite{Weil1949}, Weil stated three extremely influential conjectures, which we will review in this section.
The history of its proofs being very well known, we will refer the reader to other sources for that aspect of the story. Rather, we will state some 
intermediate points in Weil's proof of the third conjecture | the ``Riemann Hypothesis'' | in the dimension $1$ case. We mention that Weil solved the conjectures for the dimension $1$ case in \cite{Weil1948}. The first two ``general conjectures'' were solved by Artin and Grothendieck (see, e.g., \cite{Grothen1964}), and independently the first one was resolved by Dwork in \cite{Dwork}. The third and most important one was killed by Deligne in his celebrated paper \cite{Deligne1974}.

\medskip
\subsection{The Weil conjectures}

Let $X$ be of finite type (still over $\F_q$); if $x$ is a closed point, the residue field $k(x)$ is a finite extension of $\F_q$ (see, e.g., the author's second chapter in this volume); let $\mathrm{deg}(x)$ denote the degree of  this extension.
Then $\zeta_X(s) = Z(X,q^{-s})$, where $Z(X,t)$ is the power series defined by the product
\begin{equation}
Z(X,t)\index{$Z(X,t)$} := \prod_{x \in \overline{X}}\frac{1}{1 - t^{\mathrm{deg}(x)}},
\end{equation}
where $\vert X\vert$ is the set of closed points of $X$.

\subsubsection{Rationality\index{rationality}}

$Z(X,t)$ is the power series expansion of a rational function in $t$.

\subsubsection{Functional Equation\index{functional equation}}

The function $Z(X,t)$ satisfies an identity of the form
\begin{equation}
Z(X,q^{-d}t^{-1}) = \pm q^{d\chi/2}t^\chi Z(X,t),
\end{equation}
where $d = \dim(X)$ and $\chi$\index{$\chi$} is the Euler characteristic. 

\subsubsection{Riemann Hypothesis\index{Riemann Hypothesis}}

It is possible to write
\begin{equation}
Z(X,t) = \frac{P_1(t)P_3(t)\cdots P_{2d - 1}(t)}{P_0(t)P_2(t)\cdots P_{2d}(t)},
\end{equation}
where the $P_r(t)$ are polynomials with integer coefficients such that $P_0(t) = 1 - t$, $P_{2d}(t) = 1 - q^dt$, and for other $j$
we have that 
\begin{equation}
P_{j}(t) = \prod_{i = 1}^{b_j}(1 - \alpha_{ij}t),
\end{equation}
with $\vert \alpha_{ij} \vert = q^{j/2}$.

\subsection{``Roundabout proof''}
\label{secHar}

In this subsection, I want to go | tersely and sketchy | over some essential steps of Weil's proof of the Riemann Hypothesis for function fields of curves over finite fields. I will follow Shai Haran's description which is explained in \cite{Haran}. In further subsections, I will elaborate a bit more on the details (both on the level of definition and formulation).

\medskip
\subsubsection{Curves over finite fields}

Let $\mC$ be a nonsingular absolutely irreducible projective curve over a finite field $\F_p$, with $p$ a prime.

Let $f: p^{\Z} \longrightarrow \Z$ be a function of finite support. Its {\em Mellin transform}\index{Mellin transform} is
\begin{equation}
\widehat{f}(s) := \sum_{n \in \Z} f(p^n)\cdot p^{ns} \ \ n \in \mathbb{C}.
\end{equation}

We associate a divisor $\widehat{f}(A)$ to $f$ on the surface $\mC \times \mC$, defined by $\widehat{f}(A) := \sum_{n \in \Z}f(p^n)\cdot A^n$, where the $A^n$ are the Frobenius correspondences given by $A^n := \{(x,x^{p^n})\}$ and $A^{-n} = p^{-n}\cdot{(A^n)}^*$, with $n \in \mathbb{N}$, and 
$*$ denoting the involution $(x,y)^* := (y,x)$. For our divisors, the intersection theory on $\mC \times \mC$ is given by:
\begin{equation}
\Bigl\langle \widehat{f}(A),\widehat{g}(A) \Bigr\rangle = \Bigl\langle \widehat{f * g^*}(A),\mbox{Diag} \Bigr\rangle.
\end{equation}
In the latter equation, we have used the notation $g^*(p^n) := g(p^{-n})\cdot p^{-n}$, so that one calculates that 
\begin{equation}
\widehat{g^*}(s) = \widehat{g}(1 - s),
\end{equation}
and $(f*g)(p^n) := \sum_{m \in \Z} f(p^m)\cdot g(p^{n - m})$, so that 
\begin{equation}
\widehat{f * g^*}(s) = \widehat{f}(s)\cdot \widehat{g}(s).
\end{equation}
Note that $\mathrm{Diag} = A^0$.  We have that 
\begin{equation}
\Bigl\langle \widehat{f}(A),\mbox{Diag} \Bigr\rangle = \widehat{f}(0) + \widehat{f}(1) - \sum_{\zeta_{\mC}(s) = 0}\widehat{f}(s).
\end{equation}
The sum is taken over the zeros of the zeta function of $\mC$.

It is now possible to show the following, letting $h^0(f) := \dim_{\F_p}(H^0(\mC \times \mC,\mO(\widehat{f}(A))))$ (i.e., it is the dimension of the space of global sections of the line bundle $\mO(\widehat{f}(A))$):

\medskip
\begin{lemma}[Riemann-Roch]\index{Riemann-Roch Theorem}
With $\omega$ a canonical divisor on $\mC \times \mC$, we have
\begin{equation}
 h^0(f) + h^0(\omega - f) \geq \frac{1}{2}\Bigl\langle \widehat{f}(A),\widehat{f}(A) - \omega\Bigr\rangle.
\end{equation}
\end{lemma}

\medskip
\begin{lemma}[Monotoneness\index{monotoneness}]
\begin{equation}
h^0(f) \geq 0 \ \ \Longrightarrow\ \ h^0(f + g) \geq h^0(g).
\end{equation}
\end{lemma}

\medskip
\begin{lemma}[Ampleness\index{ampleness}]
\begin{equation}
\Bigl\langle \omega,f \Bigr\rangle = 0 \ \ \Longrightarrow \ \ h^0(m\cdot f)
\end{equation}
is bounded independently of $m \in \Z$.\\
\end{lemma}

(Further on we will provide more details concerning the notions used in these properties.)
Using ``Riemann-Roch,'' ``monotoneness'' and ``ampleness,'' one obtains the

\medskip
\begin{lemma}[Fundamental Inequality]\index{Fundamental Inequality}
\begin{equation}
\widehat{f}(0)\cdot\widehat{f}(1) \geq \frac{1}{2}\cdot\Bigl\langle \widehat{f}(A),\widehat{f}(A)\Bigr\rangle.
\end{equation}
\end{lemma}

This expression is equivalent to:
\begin{equation}
\sum_{\zeta_{\mC}(s) = 0}\widehat{f}(s)\cdot\widehat{f}(1 - s) \geq 0.
\end{equation}
And the latter implies 
\begin{equation}
\zeta_{\mC}(s) = 0\ \ \Longrightarrow\ \ \Re(s) = \frac{1}{2},
\end{equation}
which is the desired Riemann Hypothesis.

\medskip
\subsubsection{Characteristic $0$}

Now we turn to the field of rational numbers, $\mathbb{Q}$. 

Let $f: \R^+ \longrightarrow \R$ be a smooth function which is compactly supported, and associate to $f$ its {\em Mellin transform}\index{Mellin transform}
\begin{equation}
\widehat{f}(s) := \int_{0}^{\infty}f(x)x^s\frac{dx}{x}.
\end{equation}

Below, $\zeta^c(s)$\index{$\zeta^c(s)$} denotes the complete zeta function $\pi^{-s/2}\Gamma(\frac{s}{2})\zeta(s)$. 
Then by the functional equation $\zeta^c(s) = \zeta^c(1 - s)$, we have that 
\begin{equation}
\widehat{f}(0) + \widehat{f}(1) - \sum_{\zeta^c(s) = 0}\widehat{f}(s) = -\frac{1}{2\pi i}\oint\widehat{f}(s)d\log\zeta^c(s) =: W(f)\index{$W(f)$}.
\end{equation}

Put $f^*(x) := f(x^{-1})\cdot x^{-1}$.  We introduce the Fundamental Inequality\index{Fundamental Inequality} in this setting:

\medskip
\begin{lemma}[Fundamental Inequality]\index{Fundamental Inequality}
\begin{equation}
\widehat{f}(0)\cdot\widehat{f}(1) \geq \frac{1}{2}W(f * f^*).
\end{equation}
\end{lemma}

It  is equivalent to:
\begin{equation}
\sum_{\zeta^c(s) = 0}\widehat{f}(s)\cdot\widehat{f}(1 - s) \geq 0.
\end{equation}

Weil pointed out that this Fundamental Inequality implies, again: 
\begin{equation}
\zeta^c(s) = 0\ \ \Longrightarrow\ \ \Re(s) = \frac{1}{2}.
\end{equation}

Now for functions $f, g: \R^+ \mapsto \R$ which are smooth and compactly supported, to be thought of as representing ``Frobenius divisors'' on the nonexisting surface $\overline{\Spec(\Z)} \times \overline{\Spec(\Z)}$, one could {\em define} their intersection number as 
\begin{equation}
\Bigl\langle f, g \Bigr\rangle := W(f * g^{*}),
\end{equation}
where $W(\cdot)$ is as above, and associating to such a function $f$ a real number $h^0(f)$ which  satisfies the three properties ``Riemann-Roch,'' ``monotoneness'' and ``ampleness,'' one will find the solution of the classical Riemann Hypothesis through the characteristic $0$ version of Weil's Fundamental Inequality. 

\medskip
\subsection{Riemann-Roch}

Let $X$ be an algebraic $k$-variety. A {\em Weil divisor}\index{Weil divisor} is a finite formal sum $D = \sum_{i = 1}^kn_{Y_i}Y_i$ of codimension one closed subvarieties of $X$ (the coefficients taken in $\Z$). The set $\Div(X) = \{ D\ \vert\ D\ \mbox{Weil divisor of}\ X \}$\index{$\Div(X)$} carries the natural structure of a free abelian group, which we denote in the same way. The {\em degree}\index{degree} of a divisor $D = \sum_{i = 1}^kn_{Y_i}Y_i$ is defined as $\mathrm{deg}(D) := \sum_{i = 1}^kn_{Y_i}$\index{$\mathrm{deg}(D)$}. 

To any $g \in k(X)$, $g \ne 0$, we can associate a divisor called ``principal divisor,''\index{principal divisor} $(g)$\index{$(g)$}, defined by
\begin{equation}
(g) := \sum\nu_x(g)x,
\end{equation}
where $\nu_g(x) = \#\mbox{zeros} - \#\mbox{poles}$\index{$\nu_g(\cdot)$} of $g$ at $x$.
(Here, $x$ varies over the irreducible codimension $1$ sub $k$-varieties of $X$, and $k(X)$\index{$k(X)} is the function field of $X$.) 

Also, to any $D \in \Div(X)$ we associate a $k$-vector space $L(D)$\index{$L(D)$} defined as
\begin{equation}
L(D) := \{ f \in k(X)^\times\ \vert\ D + \divf(f) \geq 0\}\ \cup\ \{0\}.
\end{equation}
The dimension of $L(D)$ is finite, and is denoted by $\ell(D)$\index{$\ell(D)$}.

Now let $X$ be a curve. The set  $\Omega(X)$\index{$\Omega(X)$} of rational differential forms forms a $k(X)$-module of dimension $1$. If $x$ is any point of $X$, and $\{t\} \subset \mO_{X,x}$ is a basis of $\fm_x/\fm_x^2$, then $\{dt\}$ is a basis for the rational differential forms on $X$ (over $k(X)$).
Now let $\omega$ be a nonzero rational differential form on $X$. Associate a divisor $D(\omega)$\index{$D(\omega)$} to $\omega$ in the following way. For any $x \in X$, let $t$ be a local coordinate as before, and write $\omega = fdt$ for some $f \in k(X)^\times$. Then the coefficient of $[x]$ in $D(\omega)$ is $\nu_x(f)$. Such a divisor is called 
``canonical divisor,''\index{canonical divisor} and is unique up to linear equivalence.



In the next theorem, the {\em genus}\index{genus} of $X$ is defined as $H^1(X,\mO_X)$.

\begin{theorem}[Riemann-Roch\index{Riemann-Roch Theorem}]
Let $X$ be a smooth projective algebraic curve.
For any element of $\Div(X)$, we have that
\begin{equation}
\ell(D) - \ell(K - D) = \mathrm{deg}(D) - g + 1,
\end{equation} 
where $g$ is the genus of $X$.
\end{theorem}

\medskip
\begin{remark}[Intersection theory in characteristic one]{\rm 
To my knowledge, the status of intersection theory in characteristic one is: ``very immature.'' Smirnov has done seminal work (and perhaps the only work) on this subject, by trying to approximate the Hurwitz formula for the ``maps'' 
\begin{equation}
q: \Spec(\Z) \longrightarrow \mathbb{P}^1/\F_1, 
\end{equation}
with $q \in \mathbb{Q}$ | see \cite{Smirnov92}. Understanding such maps can be seen as an instance of understanding intersection theory on the surface 
\begin{equation}
\Spec(\Z) \times \mathbb{P}^1/\F_1. 
\end{equation}
We refer to Le Bruyn's chapter for an elaborate discussion.
}
\end{remark}

\subsection{Some further instances of Weil's proof}

From now on, $X$ is a nonsingular geometrically connected projective curve over $\F_q$, and $\overline{X} := X \otimes_{\F_q}\overline{\F_q}$
(with $q$ a power of the prime $p$). The zeta function $Z(X,T)$ has rational coefficients, so its zeros appear in complex conjugate pairs. The Riemann Hypothesis states that each $\omega_i$ satisfies
\begin{equation}
\frac{q}{\omega_i} = \overline{\omega_i}.
\end{equation}

The following is an equivalent reformulation of the third Weil conjecture for curves. Below, $u(x) = \mathrm{O}(v(x))$\index{$\mathrm{O}(\cdot)$}, for real functions $u, v$, means 
as usual that there is a positive constant $c$ such that $\vert u(x)\vert \leq c\vert v(x)\vert$ if $x$ is large enough.

\begin{proposition}
The Riemann Hypothesis holds for $X$ if and only if
\begin{equation}
\vert X(\F_{q^n})\vert = q^n + \mathrm{O}(q^{n/2})\ \ \ n \mapsto \infty.
\end{equation}
\end{proposition}
\begin{proof}[Sketch.]
Suppose that $\vert \omega_i \vert = q^{1/2}$ for all $i$. Start from the equality
\begin{equation}
\frac{\prod_{i = 1}^{2g}(1 - \omega_i t)}{(1 - t)(1 - qt)} = Z(X,t),
\end{equation}
and take logarithmic derivatives of both sides while multiplying by $t$ (in the ring of formal power series $\mathbb{C}[[t]]$):
\begin{equation}
t\cdot\frac{d}{dt}\Bigl(\frac{\prod_{i = 1}^{2g}(1 - \omega_i t)}{(1 - t)(1 - qt)}\Bigr) = t\cdot\frac{d}{dt}\Bigl(Z(X,t)\Bigr),
\end{equation}
to conclude that
\begin{equation}
\sum_{n \geq 1}\Bigl(1 + q^n - \sum_{i = 1}^{2g}\omega_i^n\Bigr)\cdot t^n = \sum_{n \geq 1}\Bigl\vert X(\F_{q^n})\Bigr\vert\cdot t^n. 
\end{equation}
So for each $n$ we have that $\vert X(\F_{q^n})\vert = 1 + q^n - \sum_{i = 1}^{2g}\omega_i^n = q^n + \mathrm{O}(q^{n/2})$ as $n \mapsto \infty$.\\

Conversely, suppose that the estimate holds. Then the calculation above shows that $\sum_{i = 1}^{2g}\omega_i^n = \mathrm{O}(q^{n/2})$. Now use the following elementary property:

\begin{quote}
If $\gamma_1,\ldots,\gamma_k$ are complex numbers such that $\vert \sum_{i = 1}^k\gamma_i^n \vert$ is a bounded function of $n$, then 
$\vert \gamma_i \vert \leq 1$ for all $i$,
\end{quote}
applied to $\gamma_i := \omega_iq^{-n/2}$ for all $i$ (with $k = 2g$) to see that
\begin{equation}
\vert \omega_i \vert \leq q^{n/2}\ \ \ \forall i.
\end{equation}
Now by the Functional Equation, one deduces that it is possible to order the $\omega_i$s in such a way that for all $i = 1,\ldots,2g$,
\begin{equation}
\vert \omega_i \vert = \frac{q}{\vert \omega_{2g + 1 - i} \vert} \geq q^{n/2}.
\end{equation}
The theorem follows.
\end{proof}

Note the analogy with the statement that the classical Riemann Hypothesis is equivalent to the estimate
\begin{equation}
\pi(x) = \int_{2}^x\frac{dx}{\log{x}} + \mathrm{O}(\sqrt{x}\log{x}).
\end{equation}

\medskip
\section{Counting functions and zeta functions}

In this section, we focus on the analytical side of the counting function alluded to in the introduction, motivated by the question  to define a projective ``curve'' $\mC := \overline{\Spec(\Z)}$ over $\Fun$ whose zeta function $\zeta_{\mC}(s)$
is the complete Riemann zeta function
\begin{equation}
\zeta_{\mathbb{Q}}(s) = \pi^{-s/2}\Gamma(\frac{s}{2})\zeta(s).
\end{equation}

\medskip
\subsection{The real counting distribution $N(x)$}

In \cite{ConCon}, Connes and Consani try to determine the {\em real} counting function $N(x) = N_{\mC}(x)$, with $x \in [1,\infty)$, associated to the curve $\mC$. As $N(1)$ is conjectured to take the value of the Euler characteristic of $\mC$, and since it is expected that $\mC$ has infinite genus (cf. the expression (\ref{Denform})), $N(1)$ should equal $-\infty$ (through sheaf cohomology). The function $N(x)$ should also be positive for real $x > 1$, since it should detect the cardinality of the point set of $\mC$ defined over the various extensions of $\Fun$.  

\begin{theorem}[\cite{ConCon}]
\label{announc}
\begin{itemize}
\item[{\rm (1)}]
The counting function $N(x)$ satisfying the above requirements exists as a distribution and is given by the formula
\begin{equation}
\label{form}
N(x) = x - \frac{d}{dx}\Bigl(\sum_{\rho \in Z}\mathrm{order}(\rho)\frac{x^{\rho + 1}}{\rho + 1}\Bigr) + 1,
\end{equation}
where $Z$ is the set of nontrivial zeros of the Riemann zeta function, and the derivative is taken in the sense of distributions.
\item[{\rm (2)}]
The function $N(x)$ is positive (as a distribution) for $x > 1$.
\item[{\rm (3)}]
The value $N(1)$ is equal to $-\infty$.
\end{itemize}
\end{theorem}

If $\mC$ is a nonsingular absolutely irreducible algebraic curve over the finite field $\F_q$, then its zeta function is
\begin{equation}
\zeta_{\mC}(s) = \prod_{\fp}\frac{1}{1 - N(\fp)^{-s}},
\end{equation}
where $\fp$ runs through the closed points of $\mC$ and $N(\cdot)$ is the norm map. If we fix an algebraic closure $\overline{\F_q}$ of $\F_q$ and let $m \ne 0$ be a positive integer, we have the following Lefschetz formula for the number $\vert \mC(\F_{q^m}) \vert$ of rational points over $\F_{q^m}$:

\begin{equation}
\vert \mC(\F_{q^m}) \vert = \sum_{\omega = 0}^2(-1)^{\omega}\mathrm{Tr}(\mathrm{Fr}^m \Big| H^\omega(\mC)) = 1 - \sum_{j = 0}^{2g}\lambda_j^m + q^m,
\end{equation}
where $\mathrm{Fr}$ is the Frobenius endomorphism acting on the \'{e}tale $\ell$-adic cohomology of $\mC$ ($\ell \ne p$, with $q$ a power of the prime $p$), the $\lambda_j$s are the eigenvalues of this action, and $g$ is the genus of the curve.  Writing the eigenvalues in the form $\lambda_r = q^{\rho}$ for $\rho$ a zero of the Hasse-Weil zeta function of $\mC$, we obtain

\begin{equation}
\vert \mC(\F_{q^m}) \vert = 1 - \sum_{\rho}\mathrm{order}(\rho){(q^\rho)}^m + q^m,
\end{equation}
which now has the same form as (\ref{form}).

\medskip
\subsection{Integral formula for $\frac{\partial_s\zeta_N(s)}{\zeta_N(s)}$}

Let $N(x)$ be a real-valued continuous counting function on $[1,\infty)$ satisfying a polynomial bound $\vert N(x) \vert \leq Cx^k$ for some positive integer $k$ and a fixed positive constant $C$. Then the corresponding generating function has the following form
\begin{equation}
Z(x,T) = \mathrm{exp}\Bigl(\sum_{r \geq 1}N(x^r)\frac{T^r}{r} \Bigr),
\end{equation}
and the power series $Z(x,x^{-s})$ converges for $\Re(s) > k$. The {\em zeta function} over $\Fun$ {\em associated to} $N(x)$\index{zeta function!associated to counting function} is:
\begin{equation}
\zeta_N(s)\index{$\zeta_N(s)$}:= \lim_{x \longrightarrow 1}Z(x,x^{-s})(x - 1)^\chi,
\end{equation}
where $\chi := N(1)$\index{$\chi$}.

With 
\begin{equation}
F(x,s)\index{$F(x,s)$} := \partial_s\Bigl(\sum_{r \geq 1} N(x^r)\frac{x^{-rs}}{r}\Bigr),
\end{equation}
the logarithmic derivative of $\zeta_N(s)$ is
\begin{equation}
\frac{\partial_s\zeta_N(s)}{\zeta_N(s)} = -\lim_{x \longrightarrow 1}F(x,s).
\end{equation}

The following lemma is a setup for the theorem of the previous subsection.

\begin{lemma}
For $\Re(s) > k$, we have that
\begin{equation}
\lim_{x \longrightarrow 1}F(x,s) = \int_{1}^{\infty}N(u)u^{-s}\frac{du}{u}
\end{equation}
and
\begin{equation}
\label{intform}
\frac{\partial_s\zeta_N(s)}{\zeta_N(s)} = -  \int_{1}^{\infty}N(u)u^{-s}\frac{du}{u}.
\end{equation}
\end{lemma}

\medskip
\subsection{Determining $N(x)$}

In \cite{ConCon}, the authors start from the expression (\ref{intform}) to determine the counting function $N_{\mC}(u)$ associated to the 
curve $\mC = \overline{\Spec(\Z)}$. So $N_{\mC}(u)$ should satisfy the equation
\begin{equation}
\label{countQ}
\frac{\partial_s\zeta_\mathbb{Q}(s)}{\zeta_\mathbb{Q}(s)} = -  \int_{1}^{\infty}N_{\mC}(u)u^{-s}\frac{du}{u},
\end{equation}
where $\zeta_{\mathbb{Q}}(\cdot)$ was defined in the beginning of this section.

The outcome of the calculation is the next theorem, which is a more precise form of Theorem \ref{announc}.

\begin{theorem}[\cite{ConCon}]
The tempered distribution $N_{\mC}(u)$ satisfying the equation
\begin{equation}
\frac{\partial_s\zeta_\mathbb{Q}(s)}{\zeta_\mathbb{Q}(s)} = -  \int_{1}^{\infty}N_{\mC}(u)u^{-s}\frac{du}{u}
\end{equation}
is positive on $(1,\infty)$ and is given on this interval by 
\begin{equation}
N(u) = u - \frac{d}{du}\Bigl(\sum_{\rho \in Z}\mathrm{order}(\rho)\frac{u^{\rho + 1}}{\rho + 1}\Bigr) + 1,
\end{equation}
where $Z$ is the set of nontrivial zeros of the Riemann zeta function, and the derivative is taken in the sense of distributions.
The value at $u = 1$ of the term $\mathlarger{\sum_{\rho \in Z}\mathrm{order}(\rho)\frac{u^{\rho + 1}}{\rho + 1}}$ is given by
\begin{equation}
\frac{1}{2} + \frac{\gamma}{2} + \frac{\log{4\pi}}{2} - \frac{\zeta'(-1)}{\zeta(-1)}.
\end{equation}
Here, $\gamma$ is the Euler-Mascheroni constant, which equals $-\Gamma'(1)$.
\end{theorem}

One verifies that $N_{\mC}(1) = -\infty$.

\medskip
\subsection{Absolute zeta and absolute Hurwitz functions}

We have seen in the previous subsection  that Connes and Consani investigated the absolute zeta function (of a scheme $X$ of finite type over $\Fun$) through the integral expression
\begin{equation}
\label{intexpr}
\mathrm{exp}\Bigl(\int_{1}^{\infty}N(u)u^{-s}\frac{du}{u\log{u}}\Bigr),
\end{equation}
with $N(u) = \vert X(\F_{1^{u - 1}})\vert$ a suitably interpolated counting function of the scheme $X$. The equality (\ref{intexpr}) can be obtained by integrating both sides of (\ref{countQ}) over $s$ (we omit the integration constant in (\ref{intexpr}), and refer to \cite{ConCon} for a discussion). Here, if $u$ is a positive integer, we see
$N(u)$ indeed as the number of $(\F_{1^{u - 1}})$-points of $X$, since 
\begin{equation}
\vert \F_{1^{u - 1}} \vert = \Bigl\vert \mu_{u - 1} \cup \{0\} \Bigr\vert = u.
\end{equation}
We already met this philosophy in the chapter of Manin and Marcolli (and in several other chapters in the special case that the Euler characteristic $N(1)$ represents the number of $\Fun$-points).\\

Much in the same spirit as in \cite{KuroOchi}, Kurokawa and Ochiai introduce the {\em absolute Hurwitz zeta function}\index{absolute Hurwitz zeta function}
\begin{equation}
Z_X(w;s)\index{$Z_X(w;s)$} := \frac{1}{\Gamma(w)}\int_{1}^{\infty}N(u)u^{-s}\frac{du}{u{(\log{u})}^{1 - w}}
\end{equation}
in order to get the following canonical normalization:
\begin{equation}
\zeta_X(s) = \mathrm{exp}\Bigl(\frac{\partial}{\partial w}Z_X(w;s)\Bigr\vert_{w = 0} \Bigr).
\end{equation}

For $w = 1$ we obtain that $Z_X(1;s) = - \mathlarger{\frac{\partial_s\zeta_N(s)}{\zeta_N(s)}}$.

Recall that the (classical) {\em Hurwitz zeta function}\index{Hurwitz zeta function} is defined as
\begin{equation}
\zeta(s;r)\index{$\zeta(s;r)$} := \sum_{n = 0}^{\infty}\frac{1}{(n + r)^s},
\end{equation}
with $\Re(s) > 1$ and $\Re(r) > 0$. Note that $\zeta(s;1)$ gives the Riemann zeta function.

For a function $N: (1,\infty) \longrightarrow \mathbb{C}$, use the notation
\begin{equation}
Z_N(w;s) := \frac{1}{\Gamma(w)}\int_{1}^{\infty}N(u)u^{-s}\frac{du}{u{(\log{u})}^{1 - w}}
\end{equation}
and 
\begin{equation}
\zeta_N(s) = \mathrm{exp}\Bigl(\frac{\partial}{\partial w}Z_N(w;s)\Bigr\vert_{w = 0} \Bigr).
\end{equation}

\medskip
\begin{theorem}[\cite{KuroOchi}]
Let $N(u) = \sum_{\alpha}m(\alpha)u^{\alpha}$ be a finite sum. Then we have the following:
\begin{itemize}
\item[{\rm (1)}]
$Z_N(w;s) = \sum_{\alpha}m(\alpha)(s - \alpha)^{w}$;
\item[{\rm (2)}]
$\zeta_N(s) = \prod_{\alpha}(s - \alpha)^{-m(\alpha)}$.
\end{itemize}
\end{theorem}

Consider for example the algebraic group scheme $X = \mathbf{SL}_2$. Its counting function is given by
\begin{equation}
\vert \mathbf{SL}_2(q) \vert = q^3 - q, \ \ \text{$q$ any prime power},
\end{equation}
so $N_X(u) = u^3 - u$. One calculates that $Z_X(w;s) = (s - 3)^{-w} - (s - 1)^{-w}$ and $\zeta_X(s) = \mathlarger{\frac{s - 1}{s - 3}}$.

\medskip
For functions $N, M: (1,\infty) \longrightarrow \mathbb{C}$, let $(N \oplus M)(u) := N(u) + M(u)$.
Then Kurokawa and Ochiai show that
\begin{equation}
Z_{N \oplus M}(w;s) = Z_N(w;s) + Z_M(w;s).
\end{equation}

\medskip
Let $N(u) = \sum_{\alpha}n(\alpha)u^{\alpha}$ and $M(u) = \sum_{\beta}m(\beta)u^{\beta}$ both be  finite sums. Let $(N \otimes M)(u) := N(u)M(u)$.
Then we have the following \cite{KuroOchi}:
\begin{equation}
Z_{N\otimes M}(w;s) = \sum_{\alpha,\beta}n(\alpha)m(\beta)(s - (\alpha + \beta))^{-w},
\end{equation}
 and 
\begin{equation}
\zeta_{N\otimes M}(s) = \prod_{\alpha,\beta}n(\alpha)m(\beta)(s - (\alpha + \beta))^{-n(\alpha)m(\beta)}.
\end{equation}

Other interesting results on, e.g., functional equations, can be found in \cite{KuroOchi}.

\medskip
\section{The object $\overline{\Spec(\Z)}$}

In this speculative section, we want to see $\overline{\Spec(\Z)}$ as a geometry over $\Fun$. In fact, as the multiplicative group $(\{-1,+1\},\cdot)$ is a subgroup of the monoid $(\Z,\cdot)$, we know at least that $\Spec(\Z)$ is defined even over $\F_{1^2}$. (And on the other hand, for no other (finite) positive integer $m \geq 3$, we have $\mu_m \subseteq (\Z,\cdot)$.) Although entirely trivial, this observation seems to live at the very core of this section.

\medskip
\subsection{The arithmetic surface, and $\Spec(\Z)$}

We recall Mumford's drawing\index{Mumford's drawing} of the ``arithmetic surface,'' which is by definition  the prime spectrum $\mathbb{A}^1_{\Z} = \wis{Spec}(\Z[x])$, cf. the original version of his Red Book \cite[p. 141]{Mumf67}. 

\[
\includegraphics[width=12cm]{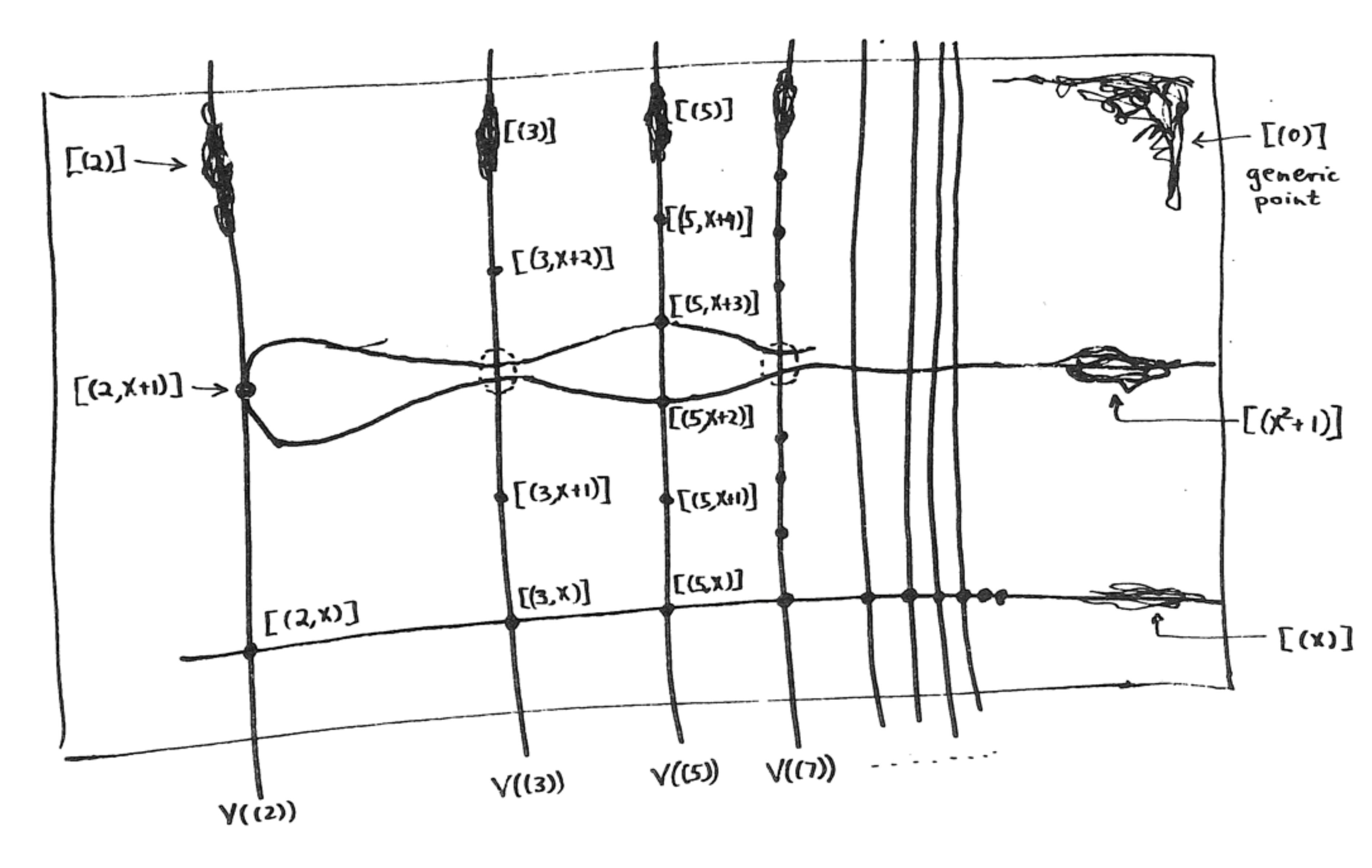} \]

As we have seen in Le Bruyn's chapter, one observes that $\wis{Spec}(\Z[x])$  contains the following elements:
\begin{itemize}
\item[{\bf Generic point}]
{$(0)$ depicted as the generic point $[(0)]$},
\item[{\bf Lines/curves}]
{principal prime ideals $(f)$, where $f$ is either a prime number $p$ (giving the vertical lines $\mathbb{V}((p)) = \wis{Spec}(\mathbb{F}_p[x])$) or a $\mathbb{Q}$-irreducible polynomial written so that its coefficients have greatest common divisor $1$ (the horizontal ``curves'' in the picture)},
\item[{\bf Intersections}]
{maximal ideals $(p,f)$ where $p$ is a prime number and $f$ is a monic polynomial which remains irreducible modulo $p$, the ``points'' in the picture.}
\end{itemize}
Mumford's drawing focuses on the vertical direction, as the vertical lines $\mathbb{V}((p))$ are the fibers of the projection 
\begin{equation}
\wis{Spec}(\mathbb{Z}[x])\ \ \mbox{\large $\twoheadrightarrow$}\ \ \wis{Spec}(\mathbb{Z}) 
\end{equation}
associated to the structural map $\Z \hookrightarrow \Z[x]$. This projection leads to Mumford's drawing of $\wis{Spec}(\Z)$ (in \cite[p. 137]{Mumf67}) where $\wis{Spec}(\Z)$ is visualized as a line:

\[
\includegraphics[width=12cm]{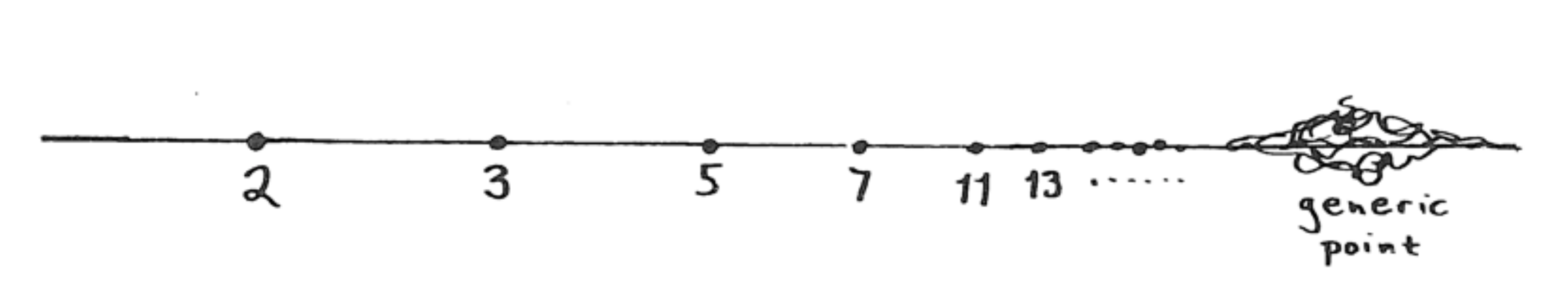} \]

The idea is that $\Z$ is a principal ideal domain like $k[x]$, $k$ a field, and 
there is one closed point for each prime number, plus a generic point $[(0)]$.

For much more details and more drawings, we refer the reader to Le Bruyn's chapter in this monograph.

\medskip
\subsection{Smirnov's $\overline{\Spec(\Z)}$}

We recall, again from Le Bruyn's chapter, the proposal due to A. L. Smirnov \cite{Smirnov92} for $\overline{\Spec(\Z)}$.\\

Smirnov proposed to take as the set of {\em schematic points}\index{schematic point} of $\overline{\wis{Spec}(\mathbb{Z})}$ the set
\begin{equation}
\{ [2],[3],[5],[7],[11],[13],[17],\hdots \}\ \cup\ \{[ \infty ]\} 
\end{equation}
of all prime numbers together with a point at infinity. The {\em degrees}\index{degree} of these schematic points were defined as
\begin{equation}
\wis{deg}([p]) = \log(p) \qquad \text{and} \qquad \wis{deg}([\infty ]) = 1. 
\end{equation}

The field of constants was defined as $\mathbb{Q} \cap \mu_{\infty} = (\{0,-1,+1\},\cdot) = \F_{1^2}$.

\medskip
\subsection{$\overline{\Spec(\Z)}$ | version 2.0}

 We follow Lorscheid \cite{LorZ} in the description below. A more general discussion on ``arithmetic curves''\index{arithmetic curve} can be found in Lorscheid's chapter.\\

In analogy with complete smooth curves over a finite field, one could expect that the underlying topological space of $X := \overline{\Spec(\Z)}$ consists of a unique generic point $\eta$, and a closed point $p$ for every (nontrivial) place $\norm\ _p$ of the ``function field'' $\mathbb{Q}$ of $\overline{\Spec(\Z)}$. 
So a closed point either is a (finite) prime $p < \infty$, or the archimedean place $p = \infty$, which is called the {\em infinite prime}\index{infinite prime}.

The closed sets of $X$ are finite sets $\{p_1,\dotsc,p_n\}$ of nontrivial places and $X$ itself. Further, there should be a structure sheaf $\cO_X$, which associates to an open set $U=X \setminus \{p_1,\dotsc,p_n\}$ the set
\begin{equation}
 \cO_X \bigl( U \bigr) \quad = \quad \left\{ \ \frac{a}{b} \in \mathbb{Q} \ \left| \ \norm{\frac{a}{b}}_q\leq 1\text{ for all }q\notin\{p_1,\dotsc,p_n\}\ \right.\right\}
\end{equation}
of regular functions. The global section is  
\begin{equation}
\Gamma(X,\cO_X) = \cO_X(X) = \{0\}\cup\mu_2,
\end{equation}
 where $\mu_2 = (\{ -1,+1\},\cdot)$ is the cyclic group of order $2$, which should be thought of as the constants of $\overline{\Spec(\Z)}$. The stalks of $\cO_X$ are given by
\begin{equation}
 \cO_{X,p} \quad = \quad \left\{ \ \frac{a}{b} \in \mathbb{Q} \ \left| \ \norm{\frac{a}{b}}_p\leq 1\ \right.\right\},
\end{equation}
with ``maximal ideals''
\begin{equation}
 \fm_p \quad = \quad \left\{ \ \frac{a}{b} \in \mathbb{Q} \ \left| \ \norm{\frac{a}{b}}_p< 1\ \right.\right\}
\end{equation}
for every prime $p \leq \infty$. 

One observes that $X$ is indeed an extension of the scheme $\Spec(\Z)$ | the restriction of $X$ to $U = X \setminus \{\infty\}$ can be identified with $\Spec(\Z)$.

One problem with this definition is that the sets $\cO_X(X \setminus \{ p_1,\ldots,p_n\})$ aren't subrings of $\mathbb{Q}$ if $\infty \not\in \{ p_1,\ldots,p_n\}$, and neither is the stalk at infinity
\begin{equation}
\cO_{X,\infty} = [ -1, + 1] \cap \mathbb{Q}.
\end{equation}
 
It is not clear as what kind of structure the sets $\cO_X(U)$ should be ``considered'' | all these sets are monoids with zero in any case. According to \cite{LorZ}, this emphasizes the viewpoint that $\overline{\Spec(\Z)}$ should be an object defined in terms of $\Fun$-geometry, whose basic idea is to forget or, at least, to loosen addition.

\medskip
\subsection{Lorscheid's blueprint product}

In Lorscheid's chapter, we have seen that in the context of blue schemes, there is a nontrivial interpretation for the object $\overline{\Spec(\Z)} \times \overline{\Spec(\Z)}$. In this subsection, we repeat some of these ideas, with more details for this specific case. We refer to the aforementioned chapter (and the references therein) for much more  details. 


\subsubsection{Blue schemes}
\label{subsection: locally blueprinted spaces}

\noindent
Denote the category of blueprints by $\bp$\index{$\bp$}.
A \emph{blueprinted space}\index{blueprinted space} is a topological space $X$ together with a sheaf $\cO_X$ in $\bp$. A \emph{morphism of blueprinted spaces}\index{morphism!of blueprinted spaces} is a continuous map together with a sheaf morphism. Since the category $\bp$ contains small colimits, the stalks $\cO_{X,x}$ in points $x\in X$ exist, and a morphism of blueprinted spaces induces morphisms between stalks. A \emph{locally blueprinted space}\index{locally blueprinted space} is a blueprinted space whose stalks $\cO_{X,x}$ are local blueprints with maximal ideal $\fm_x$ for all $x\in X$. A \emph{local morphism}\index{morphism!local} between locally blueprinted spaces is a morphism of blueprinted spaces that induces local morphisms of blueprints between all stalks. We denote the resulting category by $\bpspaces$.\index{$\bpspaces$}

Let $x$ be a point of a locally blueprinted space $X$. We define the \emph{residue field of $x$}\index{residue field} as the blue field $\kappa(x)=\cO_{X,x}/\fm_x$. A local morphism of locally blueprinted spaces induces morphisms between residue fields.

The \emph{spectrum of a blueprint $B$}\index{spectrum of a blueprint} is defined analogously as the case of rings or monoids with zero: $\Spec(B)$ is the locally blueprinted space whose underlying set $X$ is the set of all prime ideals of $B$, endowed with the Zariski topology, and whose structure sheaf $\cO_X$ consists of localizations of $B$.  A \emph{blue scheme}\index{blue scheme} is a locally blueprinted space that is locally isomorphic to spectra of blueprints. We denote the full subcategory of $\bpspaces$ whose objects are blue schemes by $\BSch$\index{$\BSch$}.

\subsubsection{Fiber products}
\label{subsection: globalizations}
\label{subsection: properties of blue schemes}

It is possible to extend some basic properties of usual schemes to blue schemes, cf. Lorscheid's chapter | we only single out the following one.
\begin{enumerate} 
 \item[{\bf Fiber}] 
 Fiber products of blue schemes exist in $\BSch$.
\end{enumerate}

In fact, the fiber products of blue schemes are of a much simpler nature than fiber products of Grothendieck schemes (in the category of Grothendieck schemes), as Lorscheid explains. This has the important effect that the fiber product in $\BSch$ coincides with the fiber product in $\bpspaces$, which is not true for Grothendieck schemes and locally ringed spaces. More precisely, the following is true.

\begin{quote}
 The category $\bpspaces$ has fiber products. The fiber product $X\times _S Y$\index{$X\times _S Y$} is naturally a subset of the topological product $X\toptimes Y$, and it carries the subspace topology. In the case of $S = \Spec(\F_{1^2})$, it has the explicit description
 \begin{equation}
  X\times_{\F_{1^2}} Y \qquad = \qquad \Bigl\{ \ (x,y)\in X\toptimes Y \ \left| \ \substack{\text{there are a semifield }k\text{ and blueprint}\\ \text{morphisms } \kappa(x)\to k\text{ and }\kappa(y)\to k} \ \right.\Bigr\}.
 \end{equation}
 
 If $X$, $Y$ and $S$ are blue schemes, then the fiber product $X\times_S Y$ in $\bpspaces$ coincides with the fiber product in $\BSch$. In particular, $X\times_S Y$ is a blue scheme.
\end{quote}

Since for every place
$p$, the residue field    $\kappa(p)$ can be embedded into $\mathbb{C}$,  the following may be concluded.

\begin{theorem}
The arithmetic surface $\Spec(\Z) \times_{\F_{1^2}} \Spec(\Z)$ is a topological space of dimension $2$.
\end{theorem}

Again, $\Spec(\F_{1^2})$ is needed.

\medskip
\subsection{$\overline{\Spec(\Z)}$ as an $\infty$-dimensional space | poor man's version}

The picture becomes much worse when one ignores addition altogether: instead of a curve, we wind up with a nasty infinite dimensional
projective space.\\

Denote the set of positive integer prime numbers as $\mP$\index{$\mP$}. Below, $\F_{1^2} := \mu_2 \cup \{0\} = (\{-1,+1\},\cdot) \cup \{0\}$.

Now define the map $\upsilon: \F_{1^2}[{X_p}_{\vert p\ \mathrm{prime}}]\  \longrightarrow\ \Z$:
\begin{equation}
\upsilon: \varrho\prod_{i \in \mP} X_i^{n_i} \longrightarrow \varrho\prod_{i \in \mP} i^{n_i},
\end{equation}
noting that on the left hand side we only consider polynomials of finite support, of course. Also, for $i \in \mP$, $n_i \in \mathbb{N}$ and $\varrho \in \F_{1^2}$. Then $\upsilon$ is a monoid isomorphism, and so $(\Z,\cdot) \cong  \F_{1^2}[{X_p}_{\vert p\ \mathrm{prime}}]$.\\

At this point, one wants to add an extra point $\infty$ to $\Spec\Bigl({\F_{1^2}[{X_p}_{\vert p\ \mathrm{prime}}]}\Bigr)$, but since the latter looks like 
an infinite dimensional affine space rather than an affine curve, we might as well add a space at infinity to $\Spec\Bigl({\F_{1^2}[{X_p}_{\vert p\ \mathrm{prime}}]}\Bigr)$ to make things more natural. And as projective spaces of the same (possibly infinite) dimension are isomorphic, we might as well go one dimension down, and do a $\wis{Proj}$-construction on ${\F_{1^2}[{X_p}_{\vert p\ \mathrm{prime}}]}$.

So, we imagine that

\begin{equation}
\overline{\Spec(\Z)}\ \ \cong\ \ \wis{Proj}\Bigl({\F_{1^2}[{X_p}_{\vert p\ \mathrm{prime}}]}\Bigr),
\end{equation}
a (countably) infinite dimensional projective space over $\F_{1^2}$. 

For every prime $p$, there is  a closed point, and the Kurokawa ($\{ \F_1,\F_{1^2}\}$-)zeta function (see \cite{Kurozeta} and the author's second chapter) should involve a factor of the form $\mathlarger{\Rprod_{i \in \{0\} \cup \mP}(s - \vartheta(i))^{-1}}$, where $\vartheta(\cdot)$ is a function which arises because we work over 
$\F_{1^2}$. I will come back to this matter in \cite{KT-cyclic}.

As we will see in the next section, we will imagine this object to be the most rigid one in a category of all possible $\Fun$-guises of $\overline{\Spec(\Z)}$.

\medskip
\section{Final speculation: the ``moduli space'' of $\overline{\Spec(\Z)}$-geometries over $\Fun$}

As we have seen in the present monograph, many approaches exist for $\Fun$-schemes, and so also for $\overline{\Spec(\Z)}$ over $\Fun$.
The coarsest is Deitmar's \cite{Deitmarschemes2} | which we will denote by $\overline{\Spec(\Z)}^{\mD}$\index{$\overline{\Spec(\Z)}^{\mD}$} in this section, and I have given its (or better, ``a'') description in the previous paragraphs. One could hence define a category $\mC(\overline{\Spec(\Z)},\Fun)$\index{$\mC(\overline{\Spec(\Z)},\Fun)$} as ``everything in between $\overline{\Spec(\Z)}^{\mD}$ and $\overline{\Spec(\Z)}$ (the latter seen as Grothendieck scheme).''

\medskip
\begin{center}
\item
$\overline{\Spec(\Z)}$\\

\item
{\Large $\uparrow$}\\

\item
the category $\mC(\overline{\Spec(\Z)},\Fun)$\\

\item
{\Large $\uparrow$}\\

\item
$\overline{\Spec(\Z)}^{\mD}$
\end{center}

\medskip
The ``in between'' relation depends on the theory. Once that theory is fixed, one imagines $\mC(\overline{\Spec(\Z)},\Fun)$ to be something like a moduli space which parametrizes (classes of) objects which descend from $\overline{\Spec(\Z)}$ to $\Fun$-schemes. In the same way, one defines  $\mC({\Spec(\Z)},\Fun)$\index{$\mC({\Spec(\Z)},\Fun)$}. 

\medskip
\subsection{Example in $\Upsilon$-scheme theory}

In $\Upsilon$-scheme theory, one would start with considering (minimal) generating sets $G = \{g_i \vert i \in I\}$ 
of $\Z$ (such as $\{3, 5\}$ or $\{6, 10, 15 \}$), and define, for each such representation,  a surjective 
homomorphism
\begin{equation}
\Phi: \mathbb{Z}{[X_i]}_I \longrightarrow \Z: X_j \longrightarrow g_j\ \ \forall j \in I, 
\end{equation}
so that  $\Z \cong \mathbb{Z}{[X_i]}_I/J$ with $J$ the kernel of $\Phi$. 
 
For an element $P$ of $J$, write $P(1)$\index{$P(1)$} for the set of ``$\mathbb{F}_1$-polynomials'' defined by $P$ as in the author's second chapter; if 
\begin{equation}
P = \sum_{i = 0}^kk_iX_0^{n_{i0}}\cdots X_m^{n_{im}}, \ \ n_{ij} \in \mathbb{N},
\end{equation}
then 
\begin{equation}
\label{P1}
P(1) := \{ X_0^{n_{i0}}\cdots X_m^{n_{im}}  \vert i = 0,\ldots,k\}.
\end{equation}
If $P$ has a nonzero constant term $c$, the corresponding element in $P(1)$ is, by definition, $0$.

The spectrum of the monoid quotient $\mathbb{F}_1[X_0,\ldots,X_m]/\langle P(1) \vert P \in J \rangle$, is a {\em bad $\mathbb{F}_1$-descent}\index{bad $\Fun$-descent} of the affine scheme $\Spec(\Z)$.\\

Then associate to $\Z$ the set 
\begin{equation}
\wis{MRep}(\Z)\index{$\wis{MRep}(\Z)$} := \{ (G,J)\ \vert\ \langle G \rangle \overset{\mbox{min}}{=} \Z, \Z \cong \Z[X_i]_{i \in G}/J\} 
\end{equation}
(= the category of minimal generating sets of $\Z$, together with explicit kernels of the natural morphism $\phi: \Z[X_i]_{i \in G} \mapsto \Z: X_g \mapsto g$).

The  elements of $\wis{MRep}(\Z)$ correspond to bad descents of $\Spec(\Z)$ as above, and isomorphism classes of the latter should define points in the ``$\F_1$-moduli space'' of $\Spec(\Z)$.

\medskip
\subsection{Final remark: zeta functions of categories}
\label{Kurocat}

As I want to see the ``space'' $\mC(\overline{\Spec(\Z)}\Big/\Spec(\Z),\Fun)$ as one object, it is desirable that one can attach a zeta function to such a space.
 In \cite{Kurokawacat}, Kurokawa introduced such an approach, as we have seen in detail in the author's second chapter. We repeat it for the sake of convenience, to end this chapter.

Let $\wis{C}$ be a category with a zero object. An object $X$ of $\wis{C}$ is {\em simple}\index{simple object} if for every object $Y$, $\mathrm{Hom}(X,Y)$ only consists of monomorphisms and zero-morphisms. The {\em norm}\index{norm}
of an object $Z$ is defined as 
\begin{equation}
N(Z) = \vert \mathrm{End}(Z,Z) \vert = \vert \mathrm{Hom}(Z,Z)\vert.
\end{equation}

An object is {\em finite}\index{finite object} if its norm is. We denote the category of isomorphism classes of finite simple objects of $\wis{C}$ by
$\mP(\wis{C})$\index{$\mP(\wis{C})$}. The {\em zeta function}\index{zeta function!of a category} of $\wis{C}$ then is
\begin{equation}
\zeta(\wis{C},s) = \prod_{P \in \mP(\wis{C})}\frac{1}{(1 - N(P)^{-s})}.
\end{equation}

\newpage
\printindex

%
%
%





\end{document}